\documentclass[12pt,a4paper]{amsart}

\usepackage{amsmath,amsfonts,amsthm,amssymb,verbatim,enumerate,graphicx}
\usepackage{enumitem}

\usepackage{color}
\usepackage{xypic}
\usepackage{hyperref}
\usepackage[marginratio=1:1,tmargin=117pt, height=600pt]{geometry}
\usepackage{graphicx}

\usepackage[flushmargin]{footmisc}
\usepackage{pdflscape}
\newtheorem{theorem}{Theorem}[section]
\newtheorem{lemma}[theorem]{Lemma}
\newtheorem{proposition}[theorem]{Proposition}
\newtheorem{corollary}[theorem]{Corollary}

\newtheorem{example}{Example}

\theoremstyle{remark}
\newtheorem*{remark}{Remark}
\newtheorem*{remarks}{Remarks}

\newcommand\qfor{\quad\text{ for }}
\newcommand \C{\mathbb{C}}

\newcommand \N{\mathbb{N}}
\newcommand \R{\mathbb{R}}
\newcommand \D{\mathbb{D}}
\newcommand \Z{\mathbb{Z}}

\newcommand*{\defeq}{\mathrel{\vcenter{\baselineskip0.5ex \lineskiplimit0pt
                     \hbox{\scriptsize.}\hbox{\scriptsize.}}}%
                     =}
\newcommand{\eqdef}{=\mathrel{\vcenter{\baselineskip0.5ex \lineskiplimit0pt
                     \hbox{\scriptsize.}\hbox{\scriptsize.}}}}
\newcommand{\red}[1]{{\color{black} #1}}
 \newcommand{\reddot}[1]{{\color{red} #1}}

\usepackage{marginnote}

\renewcommand{\subset}{\subseteq}
\newcommand\xqed[1]{%
  \leavevmode\unskip\penalty9999 \hbox{}\nobreak\hfill
  \quad\hbox{#1}}
\newcommand\exend{\xqed{$\triangle$}}
\begin{document}
\bibliographystyle{amsalpha}
%
%
%
%
\title[Connectivity of $I(f)$ in $\C^*$]{On the connectivity of the escaping set in the punctured plane}
\author[{V. Evdoridou \and D. Mart\'i-Pete \and D. J. Sixsmith}]{Vasiliki Evdoridou \and David Mart\'i-Pete \and David J. Sixsmith}

\address{School of Mathematics and Statistics\\ The Open University\\
Milton Keynes MK7 6AA\\
UK
\textsc{\newline \indent \href{https://orcid.org/0000-0002-5409-2663}{\includegraphics[width=1em,height=1em]{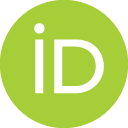} {\normalfont https://orcid.org/0000-0002-5409-2663}}}}
\email{vasiliki.evdoridou@open.ac.uk}

\address{Institute of Mathematics of the Polish Academy of Sciences\\ ul. \'Sniadeckich~8\\
00-656 Warsaw\\ Poland
\textsc{\newline \indent \href{https://orcid.org/0000-0002-0541-8364}{\includegraphics[width=1em,height=1em]{orcid2.png} {\normalfont https://orcid.org/0000-0002-0541-8364}}}}
\email{dmartipete@impan.pl}

\address{Department of Mathematical Sciences \\
	 University of Liverpool \\
   Liverpool L69 7ZL\\
   UK 
\newline \indent \href{https://orcid.org/0000-0002-3543-6969}{\includegraphics[width=1em,height=1em]{orcid2.png} {\normalfont https://orcid.org/0000-0002-3543-6969}}} 
\email{djs@liverpool.ac.uk}

\thanks{The first author was supported by Engineering and Physical Sciences Research Council grant EP/R010560/1. The second author was supported by the Japan Society for the Promotion of Science grant-in-aid 16F16807 and by the National Science Centre, Poland, grant 2016/23/P/ST1/ 04088 under the POLONEZ programme which has received funding from the EU\;\protect\includegraphics[width=.025\linewidth]{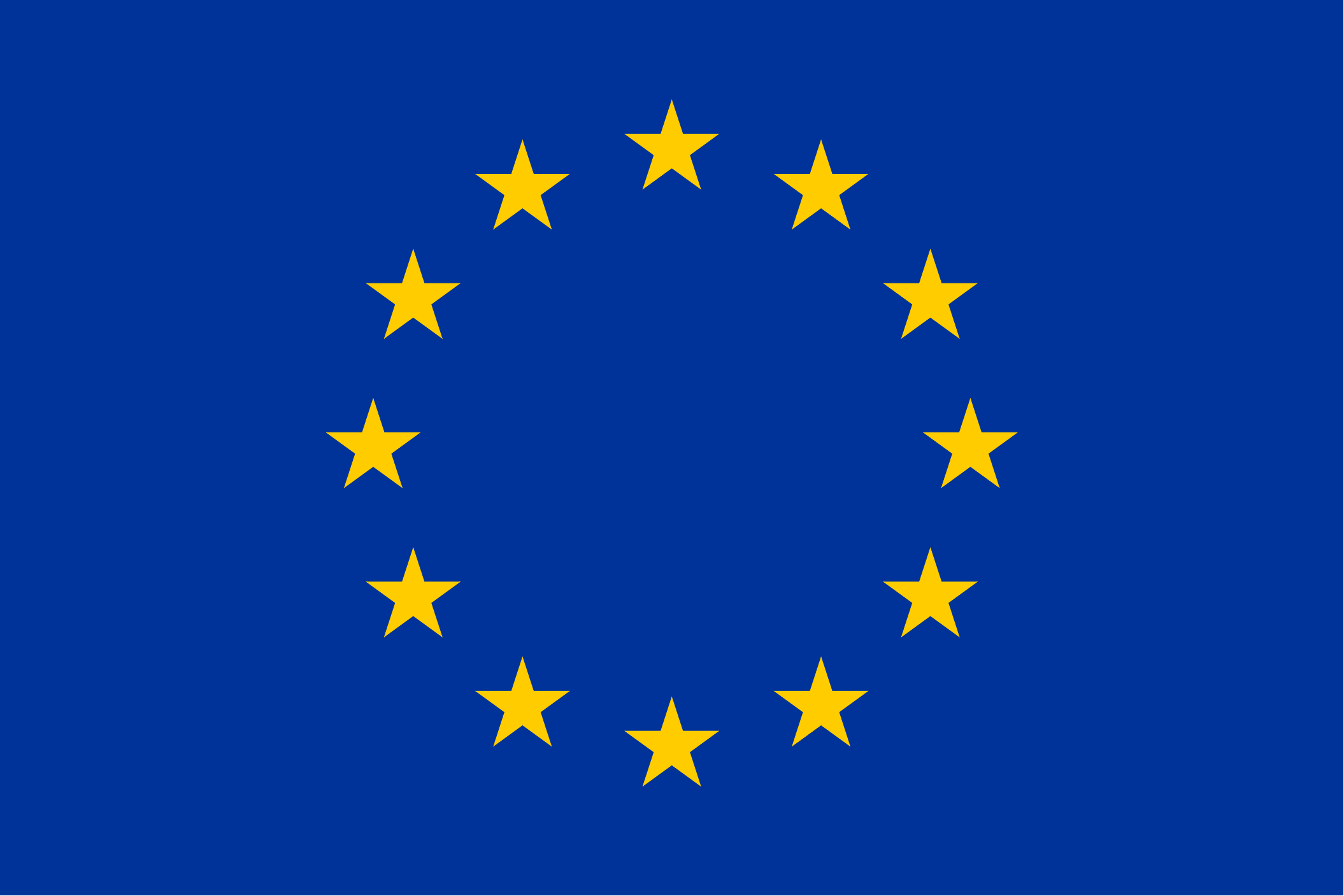} Horizon 2020 research and innovation programme under the MSCA grant agreement No. 665778.\vspace{3pt}\vspace{3pt}\vspace{3pt}\\ }
%
%
%
%
\begin{abstract}
We consider the dynamics of transcendental self-maps of the punctured plane, $\C^*=\C\setminus \{0\}$. We prove that the escaping set $I(f)$ is either connected, or~has infinitely many components. We also show that $I(f)\cup \{0,\infty\}$ is either connected, or has exactly two components, one containing $0$ and the other $\infty$. This gives a trichotomy regarding the connectivity of the sets $I(f)$ and $I(f)\cup \{0,\infty\}$, and we give examples of functions for which each case arises. 

Finally, whereas
Baker domains of transcendental entire functions are simply connected, we
show that Baker domains can be doubly connected in $\C^*$ by cons\-tructing the first
such example. We also prove that if $f$ has a doubly connected Baker domain, then its closure contains both $0$ and $\infty$, and hence $I(f)\cup\{0,\infty\}$ is connected.\vspace*{-10pt}
\end{abstract}
\maketitle
%
%
%
\section{Introduction}
Let $S$ be the complex plane, $\C$, or the punctured plane, $\C^*\defeq\C\setminus\{0\}$, and suppose that $f:S\to S$ is a holomorphic function such that $\widehat{\C}\setminus S$ consists of essential singularities of $f$, where $\widehat{\C}\defeq\C\cup\{\infty\}$ is the Riemann sphere. When $S=\C$, $f$ is a transcendental entire function, and when $S=\C^*$, we say that $f$ is a \textit{transcendental self-map of $\C^*$}. This paper concerns the iteration of this second class of functions, first studied by R{\r{a}}dstr\"om in \cite{radstrom53}. We define the \textit{Fatou set} of $f$ by
$$
F(f)\defeq\{z\in S\, :\, \{f^n\}_{n\in\mathbb{N}} \text{ is a normal family in an open neighbourhood of } z\},
$$
and we define the \textit{Julia set} of $f$ as its complement in $S$, that is, \mbox{$J(f)\defeq S\setminus F(f)$}. We use the term \textit{Fatou component} to refer to each connected component of $F(f)$. For more background and definitions, we refer to \cite{bergweiler93}.

For a transcendental entire function $f$, the \emph{escaping set} of $f$ is defined by
$$
I(f) \defeq \{z\in\C : f^n(z)\to\infty \text{ as } n\to\infty\}.\vspace*{5pt}
$$
Eremenko \cite{eremenko89} was the first to study this set in full generality. He showed that $I(f)\neq\emptyset$, $J(f)=\partial I(f)$, and that the components of $\overline{I(f)}$ are all unbounded. He conjectured that, in fact, all the components of $I(f)$ are unbounded. Although significant progress has been made on this important conjecture, it remains open, and has motivated much research on transcendental dynamics in recent years. 

It is straightforward to see that Eremenko's conjecture holds whenever $I(f)$ is connected. Because of this property and the relation between $I(f)$ and $J(f)$ discussed above, it is natural to study the connectivity of this set. Rippon and Stallard \cite[Corollary~5.1~(a)]{rippon-stallard11} (see also \cite[Theorem~1.3]{rippon-stallard18}) showed that either $I(f)$ is connected, or has infinitely many components. There exist several examples of transcendental entire functions with a connected escaping set; for example, this is the case for the exponential function \cite{rempe11}. Furthermore, for many functions $I(f)$ is a \textit{spider's web}, that is, a connected set that separates every point of $\C$ from $\infty$; see, for example, \cite{Vasso1}. Rippon and Stallard \cite{rippon-stallard11} also proved that $I(f)\cup\{\infty\}$ is a connected subset of $\widehat{\C}$; note that this does not rule out the possibility that $I(f)$ has a bounded component. 

Now, suppose that $f$ is a transcendental self-map of $\C^*$. Many authors have studied the dynamics of these maps, and shown that there are many similarities with the dynamics of transcendental entire functions, though also striking differences. In line with these studies, our principle goal in this paper is to generalise the results mentioned above to the escaping set of $f$, which in this setting is defined~as
\[
I(f) \defeq \{z\in\C^* : \omega(z,f)\subseteq \{0,\infty\}\},
\]
where $\omega(z,f) \defeq \bigcap_{n\in\mathbb{N}}\overline{\{f^k(z)\, :\, k\geqslant n\}}$, and the closure is taken in $\widehat{\C}$. This set was studied extensively in \cite{martipete1} where, in analogy with Eremenko's results, it was shown that $I(f)\neq\emptyset$, $J(f)=\partial I(f)$ and all the components of $\overline{I(f)}$ are \emph{unbounded in $\C^*$}; in other words, their closure in $\widehat{\C}$ meets $\{0,\infty\}$.

Note that, unlike the escaping set of a transcendental entire function, the escaping set of a transcendental self-map of $\C^*$ can be partitioned in a natural way into uncountably many non-empty disjoint sets that are completely invariant; recall that a set $X$ is \textit{completely invariant} when $z \in X$ if and only if $f(z) \in X$. Set $\N_0 \defeq \N \cup \{0\}$. For every $z\in I(f)$, we define the \textit{essential itinerary} of $z$ as the sequence $e=(e_n)_{n\in\N_0}\in \{0,\infty\}^{\mathbb{N}_0}$ given by, for $n\in\mathbb{N}_0$,
\[
e_n \defeq \left\{
\begin{array}{ll}
0,& \text{ if } |f^n(z)|\leqslant 1,\vspace{5pt}\\
\infty,& \text{ if } |f^n(z)|> 1.
\end{array}
\right.
\]
For each $e \in \{0, \infty\}^{\N_0}$, the set of escaping points whose essential itinerary is eventually a shift $\sigma^k(e)$ of $e$ (here $\sigma$ is the \textit{Bernoulli shift} that removes the first symbol of a sequence and moves all the other symbols one position to the left) is 
\[
I_e(f) \defeq \{z \in I(f) : \text{there exist } \ell, k \geqslant 0,\ |f^{n+\ell}(z)| > 1 \Leftrightarrow e_{n+k} = \infty \text{ for all } n \geqslant 0\}.
\]
We call $I_e(f)$ the \emph{little escaping set} with essential itinerary $e$, using the terminology from \cite{dananddave}. In the particular cases where $e$ is the constant sequence $\overline{0}$ and $\overline{\infty}$, we denote the set $I_e(f)$ by $I_0(f)$ and $I_\infty(f)$, respectively.

Mart\'i-Pete \cite{martipete1} showed that for each $e\in\{0,\infty\}^{\mathbb{N}_0}$, we have $I_e(f)\neq \emptyset$, $J(f)=\partial I_e(f)$ and all components of $\overline{I_e(f)}$ are unbounded. Note that, although there are uncountably many non-empty disjoint subsets $I_e(f)\subseteq I(f)$, several components of different sets $I_e(f)$ may lie in the same component of $I(f)$ (this is the case, for example, when $I(f)$ is connected). 

Our first result concerns the connectivity of these sets.
\begin{theorem}
\label{thm:I-connected-or-infinitely-many-comp}
Let $f$ be a transcendental self-map of $\C^*$. The set $I(f)$ is either connected or has infinitely many components. Similarly, for each~\mbox{$e \in \{0,\infty\}^{\mathbb{N}_0}$}, the set $I_e(f)$ is either connected or has infinitely many components. 
\end{theorem}

\begin{remark}
For a transcendental entire function $f$, it is an open question whether $I(f)$ necessarily has uncountably many components if it is disconnected, although it can be shown that the intersection of $I(f)$ with the complement of any disc that meets the Julia set must have uncountably many components (see \cite[Theorem~1.1~and~Theorem~1.2]{rippon-stallard17}). This is also an open question in our setting.
 \end{remark}

Our second result concerns the connectivity of the union of these sets with the set of essential singularities $\{0,\infty\}$, considered as a subset of $\widehat{\C}$.
\begin{theorem}
\label{thm:I-and-infinity}
Let $f$ be a transcendental self-map of $\C^*$. The set $I(f)\cup\{0,\infty\}$ is either connected or it consists of two components $I^0(f)$ and $I^\infty(f)$ that contain $0$ and $\infty$, respectively. Similarly, for each $e \in \{0,\infty\}^{\mathbb{N}_0}$, the set $I_e(f)\cup\{0,\infty\}$ is either connected or it consists of two components $I_e^0(f)$ and $I_e^\infty(f)$ that contain $0$ and $\infty$, respectively. Furthermore, $I^0(f)\cap I_e(f)\neq \emptyset$ and $I^\infty(f)\cap I_e(f)\neq \emptyset$ for all $e\in\{0,\infty\}^{\N_0}$.
\end{theorem}

\begin{remark}\normalfont
Note that it is easy to deduce the same connectedness properties for the \emph{fast escaping set} $A(f)$ and the \emph{little fast escaping sets} $A_e(f)$, for $e\in\{0,\infty\}^{\N_0}$, instead of $I(f)$ and $I_e(f)$, respectively, where $f$ is a transcendental self-map of $\C^*$. However, the definitions of these sets are complicated, and so we refer to \cite[Definition~1.2]{martipete1} for more details.
\end{remark}

It follows from the previous two theorems that there are three possibilities regarding the connectivity of the sets $I(f)$ and $I(f)\cup\{0,\infty\}$, namely:
\begin{enumerate}[label=(I{\arabic*})]
\item $I(f)\cup\{0,\infty\}$ and $I(f)$ are both connected;\label{I1}
\item $I(f)\cup\{0,\infty\}$ is connected and $I(f)$ has infinitely many components;\label{I2}
\item $I(f)\cup\{0,\infty\}$ has two components and $I(f)$ has infinitely many components.\label{I3}
\end{enumerate}

In Section~\ref{sec:examples}, we give several examples to show that all three cases are attained, as well as to illustrate different properties of these sets; some of the examples have appeared before in the literature, but others are new. In \cite{ems19}, we proved that the function from \cite[Example~3.3]{martipete3} has the property that its escaping set is a $\C^*$-\textit{spider's web}, that is, a connected set which separates every point of $\C^*$ from $\{0,\infty\}$, and hence is an example of type \ref{I1} (see Example~\ref{ex:1}). When $I(f)$ is disconnected, the set $I(f)\cup\{0,\infty\}$ can be connected or disconnected. We give a function $f$ such that $\mathbb{R}\setminus\{0\}\subseteq I(f)$ and $i\mathbb{R}\setminus\{0\}\subseteq \C^*\setminus I(f)$, and hence $f$~is of type \ref{I2} (see Example~\ref{ex:4} and Figure~\ref{f2}). On the other hand, to show that $I(f)\cup \{0,\infty\}$ is disconnected, that is, $f$ is a function of type \ref{I3}, it suffices to find a continuum in $\C^*\setminus I(f)$ that separates $0$ from $\infty$. We discuss two different examples of situations in which this happens. First, observe that this is the case when $f$ has a doubly connected Fatou component in $\C^*\setminus I(f)$; the function in Example~\ref{ex.disjoint-type} has a doubly connected basin of attraction. Another situation in which $I(f)\cup \{0,\infty\}$ is disconnected is when $f$ has an invariant curve around the origin; the functions in Example \ref{ex.arnold} all satisfy that the unit circle is invariant. Finally, we give an example of a function $f$~for which $I_\infty(f)$ is connected, but not a spider's web (see Example~\ref{ex.5}). 

%

We emphasise that although $I(f)$ and $J(f)$ satisfy similar properties, the connectivity of $I(f)$ is independent of that of $J(f)$. Recall that if $f$ is a transcendental entire function, then $J(f)$ is either connected or has uncountably many components \cite{baker-dominguez00} and $J(f)\cup\{\infty\}$ is connected if and only if~$f$~has no multiply connected Fatou components \cite{kisaka98}. The function $f(z)=\sin z$ is an example for which $J(f)$ is connected \cite{dominguez97} (but not a spider's web \cite{osborne13b}) and $I(f)$ is disconnected as $\R\subseteq \C\setminus I(f)$. On the other hand, for \textit{Fatou's function} $f(z)=z+1+e^{-z}$ we know that $J(f)$ is disconnected (it is an uncountable union of disjoint curves), but $I(f)$ is connected (in fact, it is a spider's web \cite{Vasso1}). For transcendental self-maps of $\C^*$, Baker and Dom\'inguez \cite[Section~3]{baker-dominguez98} proved a similar trichotomy to our cases \ref{I1}, \ref{I2} and \ref{I3} but for $J(f)$, and gave examples of functions of each case:
\begin{enumerate}[label=(J{\arabic*})]
\item $J(f)\cup\{0,\infty\}$ and $J(f)$ are both connected;\label{J1}
\item $J(f)\cup\{0,\infty\}$ is connected and $J(f)$ has uncountably many components;\label{J2}
\item $J(f)\cup\{0,\infty\}$ has two components and $J(f)$ has uncountably many components.\label{J3}
\end{enumerate}
Observe that, for example, the functions in the complex Arnol'd standard family for which $J(f)=\C^*$ are of type \ref{I3} and \ref{J1} (see Example~\ref{ex.arnold}).\vspace*{-2pt}

\begin{remarks}\normalfont
\mbox{}
\begin{enumerate} 
\item[(i)] Many authors \cite{osborne13,osborne-rippon-stallard17,dave} have studied the connectivity of other dynamically meaningful sets, such as the sets of bounded or unbounded orbits, for transcendental entire functions. It would be interesting to study the connectivity of these sets for transcendental self-maps of $\C^*$.
\item[(ii)] In \cite{dananddave} it was shown that many properties of the dynamics of transcendental self-maps of $\C^*$ carry over to \emph{quasiregular} maps of punctured Euclidean space. It is natural to ask if the connectivity results of this paper can also be transferred into this wider setting.\vspace*{-2pt}
\end{enumerate}
\end{remarks}
Let $f$ be a transcendental entire function or a  transcendental self-map of $\C^*$. Suppose that $U$ is a Fatou component of $f$, and let $U_n$ be the Fatou component containing $f^n(U)$ for $n\in\N_0$. If $U \subset U_p$ for some minimal $p \in \N$, then we say that $U$ is \emph{periodic} of period~$p$. If $U$ is not periodic, but $U_k$ is periodic for some $k \in \N$, then we say that $U$ is \emph{preperiodic}. Otherwise we say that $U$ is a \emph{wandering domain}. If $U$ is periodic and meets $I(f)$, in which case $U$ is contained in $I(f)$, then we say that $U$ is a \emph{Baker domain}. In order to prove Theorem~\ref{thm:I-and-infinity}, we need a result concerning escaping points on the boundaries of Fatou components, which may be of independent interest. Observe first that Baker domains and escaping wandering domains are the only two types of Fatou components that lie in $I(f)$. Recall that for entire functions, Rippon and Stallard \cite[Theorem~1.1]{rippon-stallard11} proved that the boundaries of escaping wandering domains always contain escaping points; the problem of whether the boundaries of Baker domains always contain escaping points remains open (see \cite{bfjk15,rippon-stallard18}). Suppose that $f$ is a transcendental self-map of $\C^*$. We consider the following subset of $I_e(f)$,\vspace*{-1pt}
$$
\widetilde{I}_e(f)\defeq\{z \in I(f) : \text{there exists } \ell \geqslant 0,\ |f^{n+\ell}(z)| > 1 \Leftrightarrow e_{n+\ell} = \infty  \text{ for all } n \geqslant 0\},\vspace*{-1pt}
$$
which has the property that if a Fatou component $U$ meets $\widetilde{I}_e(f)$, then $U\subseteq \widetilde{I}_e(f)$ (see Lemma~\ref{lem:U-meets-I}). We call $\widetilde{I}_e(f)$ the \textit{immediate little escaping set} with essential itinerary $e$, in analogy to the term \textit{immediate basin of attraction}. Observe that, in general, these sets are not completely invariant. 
Moreover, $I_e(f)=\widetilde{I}_e(f)$ if and only if the sequence $e$ only has one symbol. The result we need to prove Theorem~\ref{thm:I-and-infinity} is the following.

\begin{theorem}
\label{thm:escaping-points-in-WD-boundaries}
Let $f$ be a transcendental self-map of $\mathbb{C}^*$, and let $U$ be a wandering domain of $f$ such that $U\subseteq \widetilde{I}_e(f)$ for some $e\in\{0,\infty\}^{\mathbb{N}_0}$. Then\vspace*{-3pt}
$$
\partial U\cap \widetilde{I}_e(f)\neq\emptyset.\vspace*{-3pt}
$$
Moreover, the set $\partial U \setminus \widetilde{I}_e(f)$ has zero harmonic measure relative to $U$. In particular, the set $\partial U \setminus I_e(f)$ has zero harmonic measure relative to $U$.
\end{theorem} 

One of the striking differences between the iteration of transcendental entire functions and that of transcendental self-maps of $\C^*$ lies in the nature of their multiply connected Fatou components. Baker \cite{baker87} proved that all Fatou components of transcendental self-maps of $\C^*$ are either simply or doubly connected, and that there is at most one doubly connected Fatou component. Baker and Dom\'inguez \cite{baker-dominguez98} showed that if $U$ is a doubly connected periodic Fatou component that is bounded away from $0$ and $\infty$, then $U$ must be a \textit{Herman ring}, that is, a doubly connected domain on which the function is conjugated to an irrational rotation. 
However, there is no such restriction if $U$ is a doubly connected periodic Fatou component that is unbounded in $\C^*$. The first example of a doubly connected Fatou component in $\C^*$ was given by Baker \cite[Theorem~1.2]{baker87}, and was the basin of attraction of an attracting fixed point (see also Example~\ref{ex.disjoint-type}).

For transcendental entire functions, Baker \cite[Theorem~3.1]{baker84} proved that Baker domains are all simply connected. We construct a transcendental self-map of $\C^*$ with a doubly connected Baker domain; we are not aware of any previous such example.

\begin{theorem}
\label{thm:baker-domain}
There exists a transcendental self-map of $\C^*$ that has a doubly connected Baker domain.
\end{theorem}


Observe that every Baker domain in $\C^*$ contains a simply connected \textit{absorbing set} $H$, that is, $f(H)\subseteq H$ and for every compact set $K\subseteq U$, there exists $n\in\N_0$ such that $f^n(K)\subseteq H$ (see Lemma~\ref{lem:abs-dom}).

Finally, we prove a connection between the fact that a transcendental self-map~$f$ of $\C^*$ has a doubly connected Baker domain and the connectivity of $I(f)\cup\{0,\infty\}$. Namely, if $f$ has a doubly connected Baker domain, then $f$ cannot be of type \ref{I3}.

\begin{theorem}
\label{thm:bd-properties}
Let $f$ be a transcendental self-map of $\C^*$, and suppose that $f$ has a doubly connected Baker domain $U$. Then the closure of $U$ in $\widehat{\C}$ contains both $0$ and $\infty$ and, in particular, $I(f)\cup \{0,\infty\}$ is connected.
\end{theorem}

To prove Theorem~\ref{thm:bd-properties}, we show that if $f$ has a doubly connected Baker domain~$U$, then $\text{ind}(f)=0$ (see Lemma~\ref{lem:index}); recall that the \textit{index} of $f$, $\textup{ind}(f)$, is the index (or winding number) of the image of any positively oriented simple closed curve separating $0$ and $\infty$ with respect to $0$.\\

\noindent
\textbf{Structure.} We prove Theorem~\ref{thm:I-connected-or-infinitely-many-comp} and Theorem \ref{thm:I-and-infinity} in Sections~\ref{sec:If-connectivity} and \ref{sec:If-union-ess-sing-connectivity}, respectively. Theorem~\ref{thm:escaping-points-in-WD-boundaries} is used in the proof of Theorem~\ref{thm:I-and-infinity}, and is also proved in Section~\ref{sec:If-union-ess-sing-connectivity}. We prove some basic facts about the set $\widetilde{I}_e(f)\subseteq I_e(f)$ in the beginning of Section~\ref{sec:If-union-ess-sing-connectivity}. Examples of functions satisfying properties \ref{I1}, \ref{I2} and \ref{I3} are given in Section~\ref{sec:examples}. Finally, the study of doubly connected Baker domains is in Section~\ref{sec:baker-domains}, where we prove Theorem~\ref{thm:baker-domain} and Theorem~\ref{thm:bd-properties}.\\

\noindent
\textbf{Notation.} We denote the open ball centred at $a \in \C$ and of radius $r>0$ by
\[
D(a, r) \defeq \{ z \in \C : |z-a| < r \}.
\]
If $X\subseteq \C^*$, we denote by $\overline{X}$ and $\widehat{X}$ the closure of $X$ in $\C^*$ and $\widehat{\C}$, respectively. We always use $\partial X$ to refer to the boundary of $X$ in $\C^*$. Recall that we say that $X$ is \textit{unbounded in} $\C^*$ if $\widehat{X}\cap\{0,\infty\}\neq\emptyset$.\\

\noindent
\textbf{Acknowledgments.} We thank the Institute of Mathematics of the Polish Aca\-demy of Sciences, and the Open University, each of which hosted parts of this research.
\section{Connectivity of the escaping set}
\label{sec:If-connectivity}

In this section we prove Theorem~\ref{thm:I-connected-or-infinitely-many-comp}. We begin by giving a more general result, which is a version of \cite[Theorem~5.2]{rippon-stallard11}.

\begin{theorem}
\label{thm:E-connected-or-infinitely-many-comp}
Let $f$ be a transcendental self-map of $\C^*$, and suppose that $E\subseteq \C^*$ is a completely invariant set such that $J(f)=\overline{E\cap J(f)}$. If $E$ is not connected, then it has infinitely many components.
\end{theorem}

\begin{remark}\normalfont
Note that in $\C^*$ we only have two cases for the connectivity of the set $E$ above, whereas for entire functions \cite[Theorem~5.2]{rippon-stallard11} there is the possibility that $E$ has two components, in which case one of the components must be a singleton consisting of the only possible exceptional point. Picard's theorem implies that holomorphic self-maps of $\C^*$ do not have any exceptional points.
\end{remark}

The proof of Theorem~\ref{thm:E-connected-or-infinitely-many-comp} is based on the following key property of the Julia set, which is known as the \textit{blowing-up property} (see \cite[Theorem~4.1]{radstrom53}).
 
\begin{lemma}
\label{lem:blow-up}
Let $f$ be a transcendental self-map of $\C^*$. If $U\subseteq \C^*$ is an open set which meets $J(f)$ and $K\subseteq \C^*$ is a compact set, then there exists $n_0=n_0(K)\in\N$ such that $f^n(U)\supseteq K$ for all $n\geqslant n_0$.
\end{lemma}

We now prove Theorem~\ref{thm:E-connected-or-infinitely-many-comp}. We suppose that $E\subseteq \C^*$ is a completely invariant and disconnected set with the property that $J(f)=\overline{E\cap J(f)}$, and we need to show that $E$ has infinitely many components.

\begin{proof}[Proof of Theorem~\ref{thm:E-connected-or-infinitely-many-comp}]

Suppose, by way of contradiction, that the set $E$ is not connected but consists of finitely many components $E_1,E_2,\hdots E_m$, with $m>1$. Without loss of generality, since $J(f)=\overline{E\cap J(f)}$, there exists $z_1\in E_1\cap J(f)$. The fact that $E$ has finitely many components implies that there exists a positive number $r$ sufficiently small that $D(z_1,r)\cap (E\setminus E_1)=\emptyset$. 

By Lemma~\ref{lem:blow-up}, there exists $N\in\mathbb{N}$ such that $f^N(D(z_1,r))\cap E_j\neq \emptyset$, for each $j\in\{1,2\}$. Since the set $E$ is completely invariant, in particular, $f^{-n}(E)\subseteq E$~for all $n\in\N$, and this implies that $f^N(E_1)\cap E_j\neq\emptyset$, for $j\in\{1,2\}$. However, we know that $f^N(E_1)\subseteq E$ is a connected set. This contradiction completes the proof.
\end{proof}

Now we are able to prove Theorem~\ref{thm:I-connected-or-infinitely-many-comp}.

\begin{proof}[Proof of Theorem~\ref{thm:I-connected-or-infinitely-many-comp}]
Suppose that $e\in\{0,\infty\}^{\N_0}$. By \cite[Theorem~1.1]{martipete1}, $I_e(f)\cap J(f)\neq \emptyset$ and, by Lemma~\ref{lem:blow-up} and the complete invariance of $J(f)$, we have $J(f)=\overline{I_e(f)\cap J(f)}$. Then, the result for $I_e(f)$ follows from Theorem~\ref{thm:E-connected-or-infinitely-many-comp} by taking $E=I_e(f)$. The result for $I(f)$ follows similarly. 
\end{proof}

%
%
%
\section{Connectivity of the escaping set union zero and infinity}
\label{sec:If-union-ess-sing-connectivity}

To match the notation used in \cite{martipete1}, observe that, for $e\in\{0,\infty\}^{\N_0}$, the little escaping set with essential itinerary $e$, $I_e(f)$, can be written as the union
$$
I_e(f)=\bigcup_{\ell\in\N_0}\bigcup_{k\in\N_0} I_e^{-\ell,k}(f),
$$
where, for $\ell\in\N_0$ and $k\in\N_0$, 
$$
I_e^{-\ell,k}(f)\defeq\{z\in I(f)\ :\ |f^{n+\ell}(z)|>1 \Leftrightarrow e_{n+k}=\infty \text{ for all } n\in \N_0\}.
$$
Then, the set $\widetilde{I}_e(f)\subseteq I_e(f)$ from the introduction, which consists of the points whose essential itinerary eventually coincides with $e$, can be written as\vspace*{-1pt}
$$
\widetilde{I}_e(f)=\bigcup_{\ell\in\N_0} I_e^{-\ell,\ell}(f).\vspace*{-2pt}
$$
Observe that $f^n(\widetilde{I}_e(f))=\widetilde{I}_{\sigma^n(e)}(f)$, for $n\in\N$. Observe also that $\widetilde{I}_e(f)= \widetilde{I}_{e'}(f)$ if and only if there exists $n\in \N$ such that $\sigma^n(e)=\sigma^n(e')$; otherwise $\widetilde{I}_e(f)\cap \widetilde{I}_{e'}(f)=\emptyset$. Finally, note that\vspace*{-2pt}
$$
I(f)=\bigcup_{e\in\{0,\infty\}^{\N_0}} I_e(f)=\bigcup_{e\in\{0,\infty\}^{\N_0}} \widetilde{I}_e(f),\vspace*{-2pt}
$$
and each little escaping set $I_e(f)$ contains the immediate little escaping sets $\widetilde{I}_{\sigma^n(e)}(f)$ for $n\in\N_0$ (and all their preimages under $f$).

The reason why we are interested in this subset of $I_e(f)$ is that it is the natural set in which an escaping Fatou component lies. Note that in \cite[p. 3]{martipete3} it was observed that if a Fatou component $U$ satifies that $U\cap I_e(f)\neq \emptyset$, then $U\subseteq I_e(f)$, but the following lemma gives more precise information.

\begin{lemma}
\label{lem:U-meets-I}
Let $f$ be a transcendental self-map of $\C^*$, and suppose that $U$ is a Fatou component of $f$. If $U\cap \widetilde{I}_e(f)\neq\emptyset$  for some $e\in\{0,\infty\}^{\N_0}$, then $U\subseteq \widetilde{I}_e(f)$.
\end{lemma}
\begin{proof}

Choose $z\in U$ and let $e\in \{0,\infty\}^{\N_0}$ be such that $z\in \widetilde{I}_e(f)$. Suppose to the contrary that $U\cap (\C^*\setminus \widetilde{I}_e(f))\neq \emptyset$. Then, we can find a point \mbox{$z'\in U\cap \partial \widetilde{I}_e(f)$}. It is easy to see that the family of iterates of $f$ is not equicontinuous on any neigbhourhood of $z'$, contradicting the Arzel\`a-Ascoli theorem; this proves the lemma.
\end{proof}

Next, we prove \mbox{Theorem~\ref{thm:escaping-points-in-WD-boundaries}}, which says that if $f$ is a transcendental self-map of $\C^*$ and $U$ is a wandering domain of $f$ such that $U\subseteq \widetilde{I}_e(f)$, then $\partial U\setminus \widetilde{I}_e(f)$ has zero harmonic measure relative to $U$. To that end, we require the following lemma (see \cite[Lemma~4.1]{osborne-sixsmith16}), which is a generalisation of \cite[Theorem~1.1]{rippon-stallard11}. Here $d(z,w)$ denotes the spherical distance between two points $z,w\in \widehat{\C}$. If $G\subseteq \C^*$ is a domain and $E\subseteq \partial G$ is a Borel set, then $\omega(z,E,G)$ denotes the harmonic measure of $E$ relative to $G$ at a point $z\in G$ (see \cite{garnett-marshall05} for a precise definition). If $\omega(z,E,G)=0$ for some $z \in G$ and hence all $z \in G$, then we say that $E$ has \textit{harmonic measure zero} relative to $G$.

\begin{lemma}
\label{lem:boundary}
Let $(G_n)_{n\in\N_0}$ be a sequence of disjoint simply connected domains~in~$\widehat{\C}$. Suppose that, for each $n\in \N$, $g_n:\overline{G}_{n-1}\to \overline{G}_n$ is analytic in $G_{n-1}$, continuous in $\overline{G}_{n-1}$, and satisfies $g_n(\partial G_{n-1})\subseteq \partial G_n$. Set
$$
h_n\defeq g_n\circ \cdots \circ g_2\circ g_1,\quad \text{ for } n\in\N.
$$
Suppose that there exist $\xi\in \widehat{\C}$, $\rho\in (0,1)$ and $z_0\in G_0$ such that
$$
d(h_{n}(z_0),\xi)<\rho,\quad \text{ for } n\in \N.
$$
Suppose finally that $c>1$, and let
$$
H\defeq \{z\in \partial G_0\ :\ d(h_{n}(z),\xi)\geqslant c\rho \text{ for infinitely many values of } n\in \N\}.
$$
Then $H$ has harmonic measure zero relative to $G_0$.
\end{lemma}

We now prove Theorem~\ref{thm:escaping-points-in-WD-boundaries}.

\begin{proof}[Proof of Theorem~\ref{thm:escaping-points-in-WD-boundaries}]
Suppose that $f$ is a transcendental self-map of $\C^*$ with a wandering domain $U\subseteq \widetilde{I}_e(f)$ for some $e\in\{0,\infty\}^{\N_0}$. It suffices to prove that $\partial U\setminus\widetilde{I}_e(f)$ has zero harmonic measure relative to $U$. For $n\in \N_0$, let $U_n$ be the Fatou component containing $f^n(U)$. 

By a result of Baker \cite{baker87}, there is at most one value $N\in\N_0$ such that $U_{N}$ is doubly connected, and $U_n$ is simply connected for all $n\neq N$. It follows from \cite[Theorem 4.3.8]{ransford95} that if $n \in \N$ and $E \subset \partial U_n$ has zero harmonic measure relative to $U_n$, then $f^{-n}(E) \cap \partial U$ has zero harmonic measure relative to $U$, regardless of whether $U$ is simply or doubly connected. We can assume, therefore, that $U_n$ is simply connected for $n \in\N_0$.

Suppose that there is a sequence  $(n_k)_{k\in \N}\subseteq \N$ such that $e_{n_k}=\infty$ for all $k\in\N$; otherwise the sequence $e$ is eventually the constant sequence $0$ and the argument below can be applied with $\xi=0$. We apply Lemma~\ref{lem:boundary} with $\xi=\infty$, $G_0=U$ and, for $k\in\N$, $G_k=U_{n_k}$ and $g_k=f^{n_k-n_{k-1}}$\red{, where $n_0=0$}, so that \red{$h_k=f^{n_k}$}. We obtain a subset $H\subseteq \partial U$, of harmonic measure zero relative to $U$, such that $\partial U\setminus \widetilde{I}_e(f) \subseteq H$; recall that if the essential itineraries of two points disagree on an infinite sequence, then the two points cannot lie in the same set $\widetilde{I}_e(f)$ for $e\in\{0,\infty\}^{\N_0}$. In particular, the set $\partial U\setminus \widetilde{I}_e(f)$, and hence also the subset $\partial U\setminus I_e(f)$, has zero harmonic measure relative to $U$ as required.
\end{proof}

If $f$ is a transcendental self-map of $\C^*$, then there exists a transcendental entire function $\tilde{f}$ such that $\exp \circ \tilde{f}\equiv f\circ\, \exp$; the function $\tilde{f}$ is called a \textit{lift}~of~$f$. If~$U$~is a wandering domain of $f$, then every component $V$ of $\exp^{-1}(U)$ is a wandering domain of~$\tilde{f}$ by \cite[Lemma~2.4]{martipete3}, and since $\exp^{-1}I(f)\subseteq I(\tilde{f})$, we have $V\subseteq I(\tilde{f})$. In this case, \cite[Theorem~1.1]{rippon-stallard11} implies that the set $\partial V$ intersects $I(\tilde{f})$, but Theorem~\ref{thm:escaping-points-in-WD-boundaries} gives more precise information, namely that $\partial V$ contains points $z$ such that $|\text{Re}\,\tilde{f}^n(z)|\to+\infty$ as $n\to \infty$; see the following corollary. 

\begin{corollary}
Let $\tilde{f}$ be a transcendental entire function that is a lift of a trans\-cendental self-map $f$ of $\C^*$. Suppose that $U\subseteq \widetilde{I}_e(f)$ is a wandering domain of $f$ and $V$ is a component of $\exp^{-1}(U)$. Then, apart possibly from a set of harmonic measure zero relative to $V$, points $z\in\partial V$ satisfy that, for any $R>0$, 
\begin{itemize}
\item $\textup{Re}\,\tilde{f}^n(z)> R$, if $e_n=\infty$, and 
\item $\textup{Re}\,\tilde{f}^n(z)<-R$, if $e_n=0$,
\end{itemize}
for all sufficiently large values of $n\in \N_0$.
\end{corollary}

At this point we require some properties of the fast escaping sets $A_e(f)\subseteq I_e(f)$, for $e\in\{0,\infty\}^{\N_0}$, however their exact definition, which is quite complicated, is not important for the matter of this paper. So we refer to \cite[Definition~1.1]{martipete1} and in the following lemma we summarise the properties that we need (see \cite[Theorem~1.1,~Theorem~1.3 and~Theorem~1.5]{martipete1}).
 
 \begin{lemma}
 \label{lem:fast-escaping-set}
 Let $f$ be a transcendental self-map of $\C^*$. For each $e\in\{0,\infty\}^{\N_0}$, there exists a set $A_e(f)\subseteq I_e(f)$ such that
 \begin{enumerate}
 \item[(i)] $A_e(f)\cap J(f)\neq \emptyset$;
 \item[(ii)] $J(f)=\partial A_e(f)$;
 \item[(iii)] all the components of $A_e(f)$ are unbounded in $\C^*$.
 \end{enumerate}
Let $A(f)$ be the union of all $A_e(f)$ for $e\in\{0,\infty\}^{\N_0}$. Then, $A(f)\cap J(f)\neq \emptyset$, $J(f)=\partial A(f)$ and all the components of $A(f)$ are unbounded in $\C^*$.
 \end{lemma}

The next lemma is similar to \cite[Lemma~4.1]{rippon-stallard11}. Here, we use the properties of the fast escaping set, and obtain that, in our setting, $\partial G\cap I(f)$ consists of uncountably many points. 

\begin{lemma}
\label{lem:blow-up-consequence}
Let $f$ be a transcendental self-map of $\C^*$. If $G\subseteq \C^*$ is a domain that is bounded in $\C^*$ and such that $G\cap J(f)\neq \emptyset$, then $\partial G\cap I_e(f)\neq\emptyset$ for all $e\in\{0,\infty\}^{\mathbb N_0}$. In particular, $\partial G$ contains uncountably many points of $I(f)$.
\end{lemma}
\begin{proof}
Let $f$ be a transcendental self-map of $\C^*$ and suppose that $G\subseteq \C^*$ is a domain that is bounded in $\C^*$ and such that $G\cap J(f)\neq \emptyset$. Lemma~\ref{lem:fast-escaping-set}~(ii) implies that $G$~intersects every set $A_e(f)$ for $e\in\{0,\infty\}^{\mathbb{N}_0}$. But, by Lemma~\ref{lem:fast-escaping-set}~(iii), the components of $A_e(f)$ are all unbounded and $G$ is a bounded domain, so\linebreak $\partial G\cap A_e(f) \neq \emptyset$ for all $e\in\{0,\infty\}^{\mathbb{N}_0}$. Hence, $ \partial G\cap I_e(f) \neq \emptyset$ for all $e\in\{0,\infty\}^{\mathbb{N}_0}$, as required. The last part of Lemma~\ref{lem:blow-up-consequence} holds because there are uncountably many disjoint sets $A_e(f)$ for $e\in\{0,\infty\}^{\mathbb{N}_0}$. Note that this follows from Lemma~\ref{lem:fast-escaping-set}~(i) and the fact that the set of sequences $e\in\{0,\infty\}^{\N_0}$ for which the sets $I_e(f)$ are disjoint is uncountable (see \cite[Remark~3.1(ii)]{martipete1}).
\end{proof}

We now give two corollaries of Theorem~\ref{thm:escaping-points-in-WD-boundaries}, and prove Theorem~\ref{thm:I-and-infinity}. These three results are the analogues in $\C^*$ of \cite[Theorem~4.1]{rippon-stallard11}. However, note that they concern $I_e(f)$ instead of the whole of $I(f)$, so in some sense they are more precise, as a component of $I(f)$ may comprise components of several $I_e(f)$.

\begin{corollary}
\label{cor:bdd-comp-of-I-meet-J}
Let $f$ be a transcendental self-map of $\C^*$, and suppose that\linebreak \mbox{$e\in\{0,\infty\}^{\N_0}$}. Then, any component of $I_e(f)$ that is bounded in $\C^*$ meets $J(f)$.
\end{corollary}
\begin{proof}
Suppose to the contrary that $G$ is a component of $I_e(f)$, with $e\in\{0,\infty\}^{\N_0}$, that is bounded in $\C^*$ and does not meet $J(f)$. Let $U$ be the Fatou component containing $G$. By Lemma~\ref{lem:U-meets-I}, $U \subset I_e(f)$, and so $U = G$. Since $G$ is bounded in~$\C^*$, $G$~must be a wandering domain. Theorem~\ref{thm:escaping-points-in-WD-boundaries} then implies that $\partial G$ contains points in $I_e(f)$, contradicting the fact that $G$ is a component of $I_e(f)$.
\end{proof}

\begin{corollary}
\label{cor:bdd-meet-I-bdary-meet-I-too}
Let $f$ be a transcendental self-map of $\C^*$. If $G\subseteq \C^*$ is a domain bounded in $\C^*$ such that $G\cap I_e(f)\neq \emptyset$ for some $e\in\{0,\infty\}^{\mathbb{N}_0}$, then $\partial G\cap I_e(f)\neq\emptyset$.
\end{corollary}
\begin{proof}
Suppose, by way of contradiction, that $G\subseteq \C^*$ is a domain that is bounded in $\C^*$, $G\cap I_e(f)\neq \emptyset$ and $\partial G\cap I_e(f)=\emptyset$. Let $G'\subseteq G$ be a component of $I_e(f)$. By Corollary~\ref{cor:bdd-comp-of-I-meet-J}, $G'\cap J(f)\neq \emptyset$. Then, it follows from Lemma~\ref{lem:blow-up-consequence} that $\partial G\cap I_e(f)\neq\emptyset$. This proves the corollary.
\end{proof}

We finish this section by proving Theorem~\ref{thm:I-and-infinity}, which says that for a transcendental self-map $f$ of $\C^*$, the set $I(f)\cup\{0,\infty\}$ is either connected or it has two connected components, one containing $0$ and one containing $\infty$. In other words, we prove that $I(f)\cup\{0,\infty\}$ does not have any component that is bounded in $\C^*$.

\begin{proof}[Proof of Theorem~\ref{thm:I-and-infinity}]
Suppose to the contrary that $G\subseteq \C^*$ is a component of $I_e(f)\cup\{0,\infty\}$ that does not meet $\{0,\infty\}$. Then, there exist three disjoint open sets $H,H_0,H_\infty\subseteq \widehat{\C}$ with
$$
I_e(f)\cup\{0,\infty\}\subseteq H\cup H_0\cup H_\infty\quad \text{ and } \quad 0\in H_0,\  \infty\in H_\infty, \ G\subseteq H.
$$ 
In particular, the set $H$ is bounded away from $0$ and $\infty$ and $H\cap I_e(f)\neq \emptyset$. Then, by Corollary~\ref{cor:bdd-meet-I-bdary-meet-I-too}, $\partial H\cap I_e(f)\neq \emptyset$, which is a contradiction. 

Observe that if $I(f)\cup\{0,\infty\}$ has two components $I^0(f)$ and $I^\infty(f)$ containing the points $0$ and $\infty$, respectively, then 
$$
I^0(f)\cap I_e(f)\neq \emptyset\quad \text{ and } I^\infty(f)\cap I_e(f)\neq \emptyset,\quad \text{ for all } e\in\{0,\infty\}^{\N_0}.
$$
This follows from the fact that for every $e\in\{0,\infty\}^{\N_0}$, by Lemma~\ref{lem:fast-escaping-set}, the set~$A_e(f)$ is non-empty and, by Picard's theorem, $A_e(f)$ meets both sets $H_0$ and $H_\infty$. Therefore Theo\-rem~\ref{thm:I-and-infinity} is proved.
\end{proof}

%
%
%
\section{Examples}
\label{sec:examples}
In this section we give examples of transcendental self-maps of $\C^*$ to show that all three cases \ref{I1}, \ref{I2} and \ref{I3} regarding the connectivity of $I(f)$ and $I(f)\cup\{0,\infty\}$ are possible. Recall from the introduction that these are
\begin{enumerate}[label=(I{\arabic*})]
\item $I(f)\cup\{0,\infty\}$ and $I(f)$ are both connected;
\item $I(f)\cup\{0,\infty\}$ is connected and $I(f)$ has infinitely many components;
\item $I(f)\cup\{0,\infty\}$ has two components and $I(f)$ has infinitely many components.
\end{enumerate}
Note that every transcendental self-map~$f$~of~$\C^*$ is of the form
$$f(z)=z^n\exp (g(z)+h(1/z)),$$ 
where $g,h$ are non-constant entire functions and $n=\textup{ind}(f)\in \Z$.

In the first example, we give a transcendental self-map of $\C^*$ of type \ref{I1}, that is, a function $f$ for which $I(f)$ is connected, and hence also $I(f)\cup\{0,\infty\}$ is connected.
\begin{example}\normalfont
\label{ex:1}
In \cite[Example~3.3]{martipete3}, it was proved that for sufficiently large values of $\lambda>0$, the function
$$f(z)\defeq \lambda z \exp (e^{-z}/z)$$
has an invariant Baker domain in which points escape to $\infty$. This was the first explicit example of a transcendental self-map of $\C^*$ with a Baker domain. Later, in \cite[Theorem~1.5]{ems19}, we showed that for this family of functions, provided that \mbox{$\lambda>0$} is large enough, $I(f)$ has the structure of a $\C^*$-spider's web; recall that a connected set $X\subseteq \C^*$ is called a $\C^*$-spider's web if it separates every point of the punctured plane from $\{0, \infty\}$.
\exend
\end{example}
The second example we give is a transcendental self-map $f$ of $\C^*$ that satisfies property \ref{I2}, that is, with a disconnected $I(f)$, but $I(f)\cup\{0,\infty\}$ being connected.

\begin{example}\normalfont
\label{ex:4}
Consider the function
\[
f(z) \defeq 2z \exp\bigl(z^2 + e^{-1/z^{4}}\bigr).
\]
Note that if $z \in \R \setminus \{0\}$, then the term in the main exponential is positive and so the main exponential term is greater than one. Hence $\R \setminus \{0\} \subset I(f)$ and, by Theorem~\ref{thm:I-and-infinity}, the set $I(f) \cup \{0,\infty\}$ is connected \red{(see Figure~\ref{f2})}.

It remains to show that $I(f)$ is not connected for this map. We consider the dynamics on the imaginary axis (minus the origin). Observe that
\[
f(iy) = 2iy \exp\bigl(-y^2 + e^{-1/y^{4}}\bigr) \eqdef i \hat{f}(y),
\]
so we need to consider the action of $\hat{f}$ on the real line (see Figure~\ref{f1}).\\

\begin{figure}[ht!]
	\includegraphics[width=.70\linewidth]{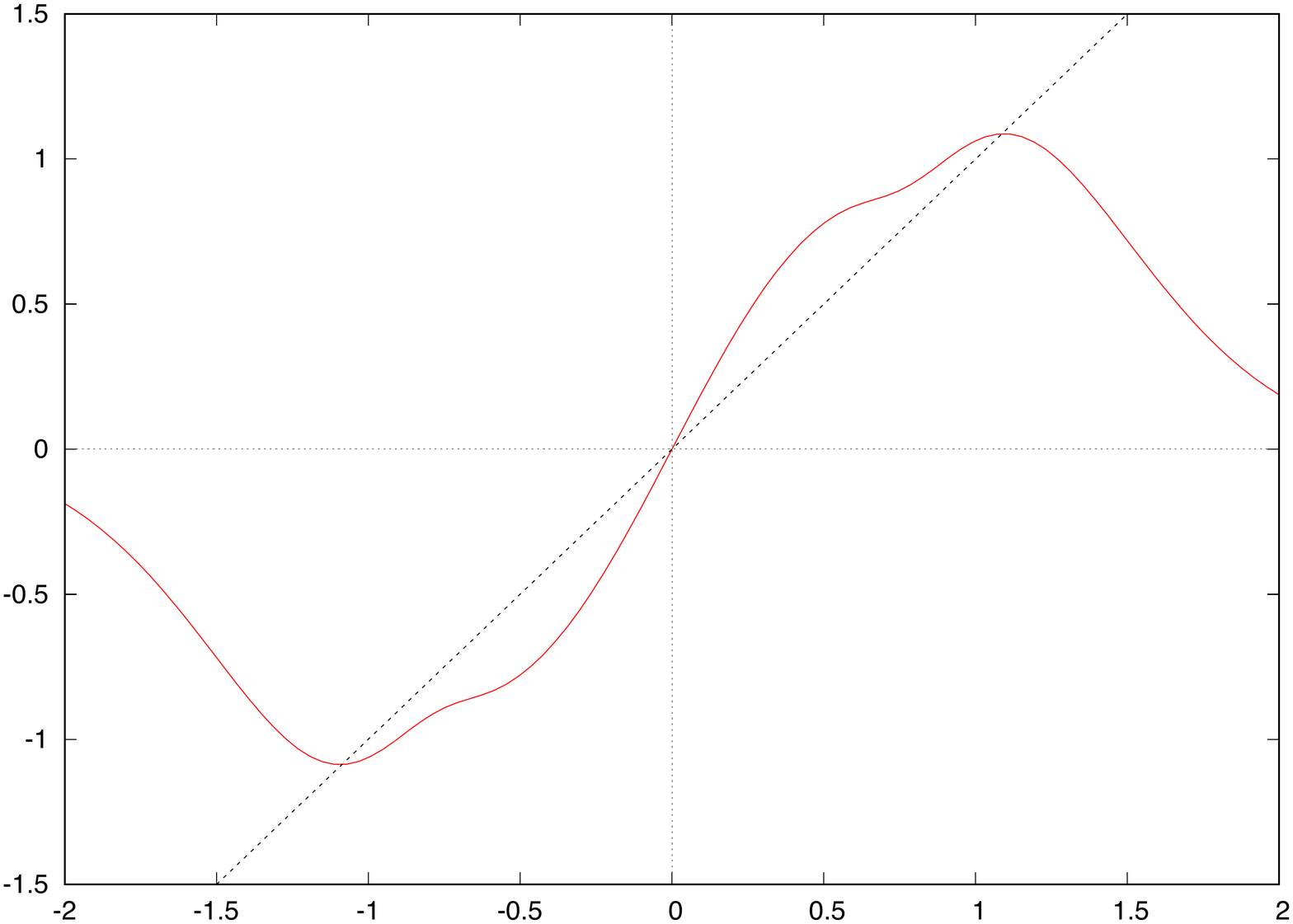}
  \caption{An illustration of the function $\hat{f}$ from Example~\ref{ex:4}.}\label{f1}
\end{figure}

Note that $\hat{f}$ has a unique attracting fixed point on the positive real line at a point $y_0  \approx 1.087$, and it is bounded on the real line. We can deduce that $f$ has two attracting fixed points at $\pm y_0i$, and that $\{ 0 + iy : y > 0 \}$ and $\{ 0 + iy : y < 0 \}$ each lie in an attracting Fatou component. Hence the imaginary axis separates $I(f)$, and the result follows.
\exend
\end{example}

\begin{figure}[ht!]
\vspace*{10pt}
	\includegraphics[width=.60\linewidth]{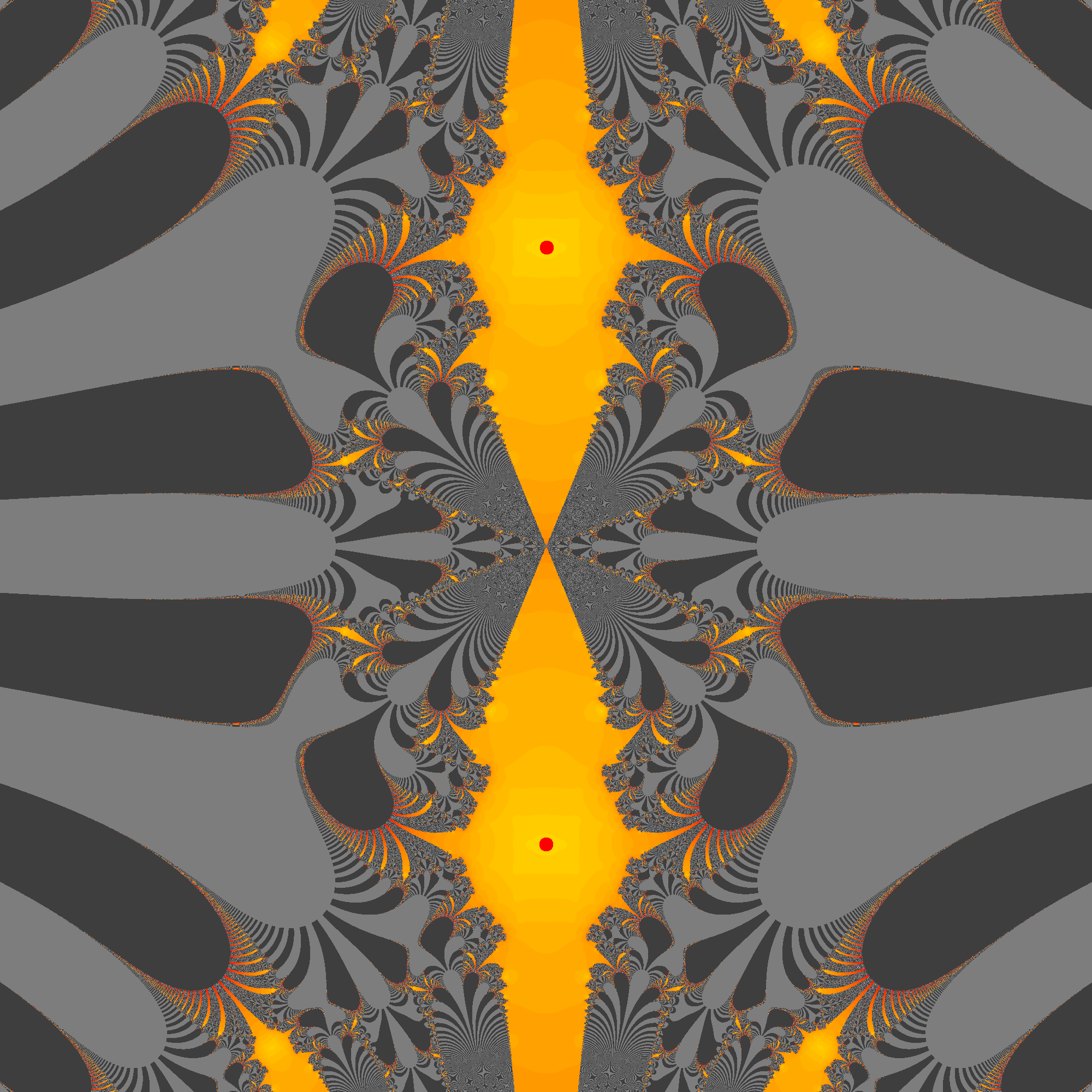}
  \caption{An illustration of the dynamics of the function $f$ from Example~\ref{ex:4}. Escaping points are coloured in grey, and yellow points lie in the basins of attraction of two attracting fixed points $\pm i y_0$ (\reddot{$\bullet$}), which contain the two halves of the imaginary axis.}\label{f2}
\end{figure}

Next, we give examples of transcendental self-maps $f$ of $\C^*$ satisfying property~\ref{I3}. Recall that, for such maps, $I(f)\cup\{0,\infty\}$  is disconnected, and hence also $I(f)$ is disconnected. We give two different situations in which this happens.

First, observe that if $f$ has a doubly connected Fatou component in $\C^*\setminus I(f)$, then $f$ is of type \ref{I3}. One class of functions with a doubly connected Fatou component is that of the so-called functions of disjoint type. We say that a trans\-cendental self-map $f$ of $\C^*$ is of \textit{disjoint type} if $F(f)$ is connected and consists of the immediate basin of attraction of a fixed point, which is doubly connected in $\C^*$ (see \cite[Definition~3.10]{fagella-martipete}). In the next example, we give a disjoint-type transcendental self-map of $\C^*$, which satisfies the additional property that each component of $I(f)$ is contained in a single set $I_e(f)$ for some $e\in\{0,\infty\}^{\N_0}$.


%
\begin{example}\normalfont
\label{ex.disjoint-type}
The function defined by
\[
f(z) \defeq \exp(0.3(z + 1/z))
\]
is a transcendental self-map of $\C^*$ of disjoint type (see \cite[Example~3.12]{fagella-martipete}) and hence satisfies \ref{I3}. Indeed, it can be shown that there exists a round annulus $A$ separating $0$ from $\infty$ that maps compactly inside itself, and this implies that $f$ has an attracting fixed point  $\alpha\in A$ \red{(see Figure~\ref{f3})}. 

Since~$f$~has \textit{finite order} as a transcendental self-map of $\C^*$ (see \cite[Definition~4.1]{fagella-martipete}), it follows from \cite[Theorem~1.7]{fagella-martipete} that each component $X$ of $I(f)$ is a curve that joins a finite point to either $0$ or $\infty$ and is contained in a single little escaping set $I_e(f)$ for some $e \in \{0, \infty\}^{\N_0}$. This contrasts with the situation where $I(f)$ is connected, and hence all sets $I_e(f)$ lie in the same component of $I(f)$; note that there are uncountably many non-empty disjoint sets $I_e(f)$ for $e\in\{0,\infty\}^{\N_0}$. 
\exend
\end{example}

\begin{figure}[ht!]
\vspace*{10pt}
	\includegraphics[width=.49\linewidth]{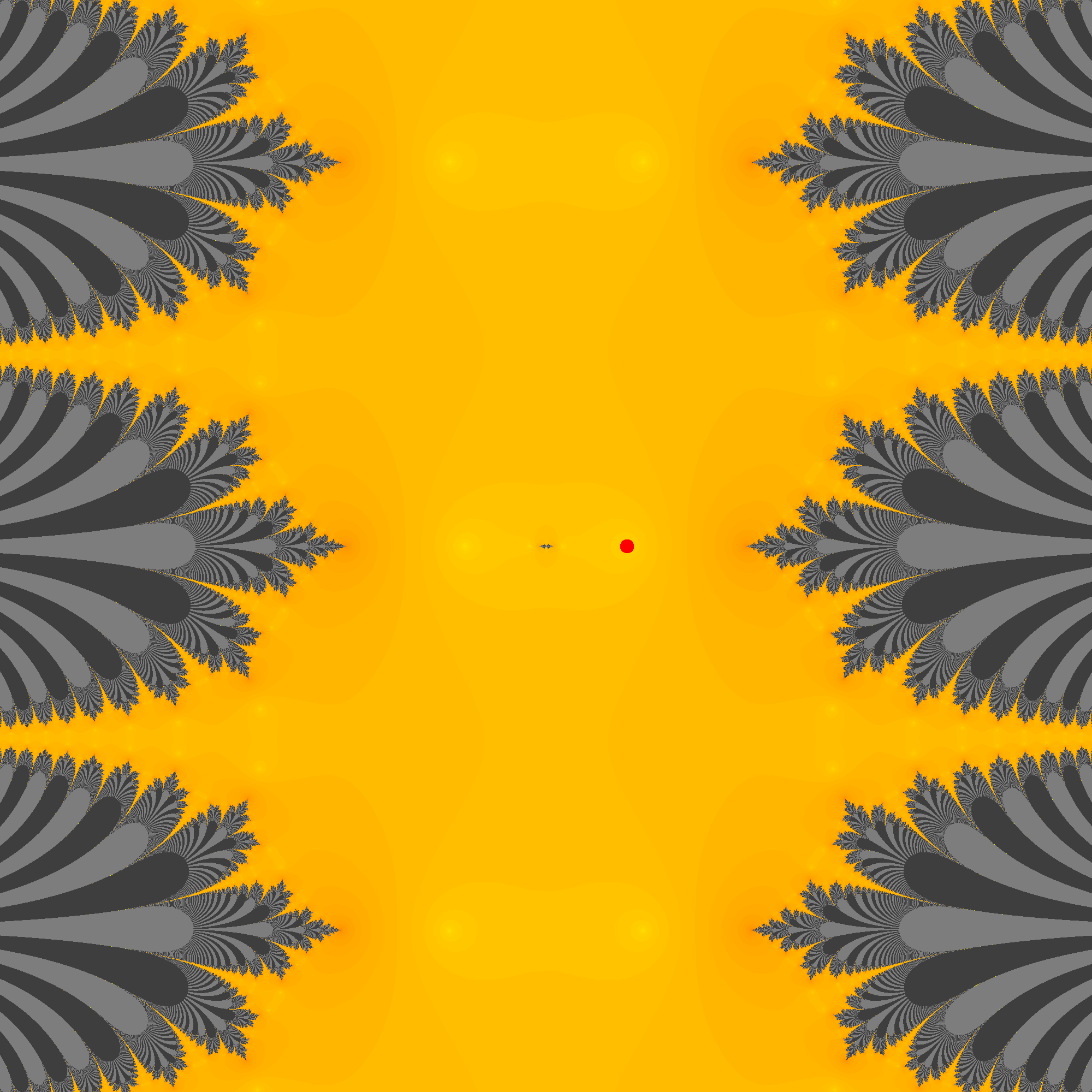} \hspace*{\fill}
	\includegraphics[width=.49\linewidth]{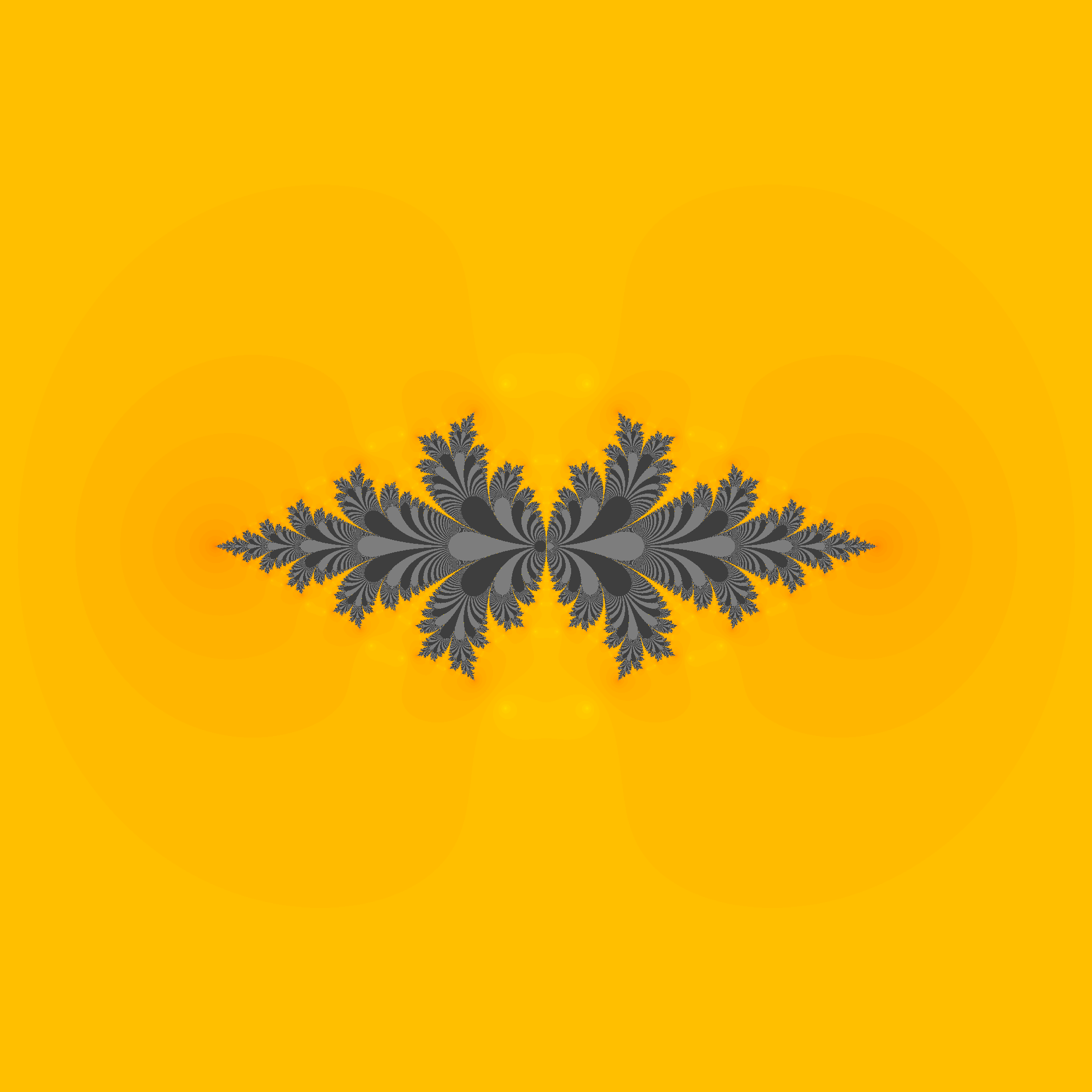}
  \caption{An illustration of the dynamics of the function $f$ from Example~\ref{ex.disjoint-type}. Escaping points are coloured in grey, and yellow points lie in the basin of attraction of an attracting fixed point (\reddot{$\bullet$}) in $\mathbb{R}$, which is doubly connected. In the right, a zoom of the origin.}\label{f3}
\end{figure}

Other examples of functions with doubly connected Fatou components in the complement of $I(f)$, such as Herman rings, were given by Baker and Dom\'inguez; see the exposition in \cite[Section 5]{baker-dominguez98}.

Next, we give a different situation in which $I(f)\cup \{0, \infty\}$ is disconnected, and hence $f$ is of type \ref{I3}. This is the case when $f$ has a forward invariant closed curve around the origin. Note that a transcendental self-map of $\C^*$ can only have one such curve. In the next example we study a well-known family with this property.

\begin{example}\normalfont
\label{ex.arnold}
The complex Arnol'd standard family is given by
\[
f_{\alpha,\beta}(z)\defeq ze^{i\alpha}e^{\beta(z-1/z)/2}, \qfor 0\leqslant \alpha \leqslant 2\pi \text{ and } \beta \geqslant 0,
\]
and the iteration of this family of transcendental self-maps of $\C^*$ was originally studied by Fagella \cite{fagella99}. Since $f_{\alpha, \beta}(\partial \D) \subset \partial \D$, we have that $\partial \D \cap I(f_{\alpha,\beta})= \emptyset$. Thus, for any $0\leqslant \alpha \leqslant 2\pi$ and $\beta \geqslant 0$, the function $f_{\alpha,\beta}$ satisfies property \ref{I3} \red{(see Figure~\ref{f4})}. Note that for some parameters $J(f_{\alpha,\beta})=\C^*$, but otherwise $f_{\alpha,\beta}$ can have a Herman ring, or other types of Fatou components.
\exend
\end{example}

\begin{figure}[ht!]
\vspace*{10pt}
	\includegraphics[width=.60\linewidth]{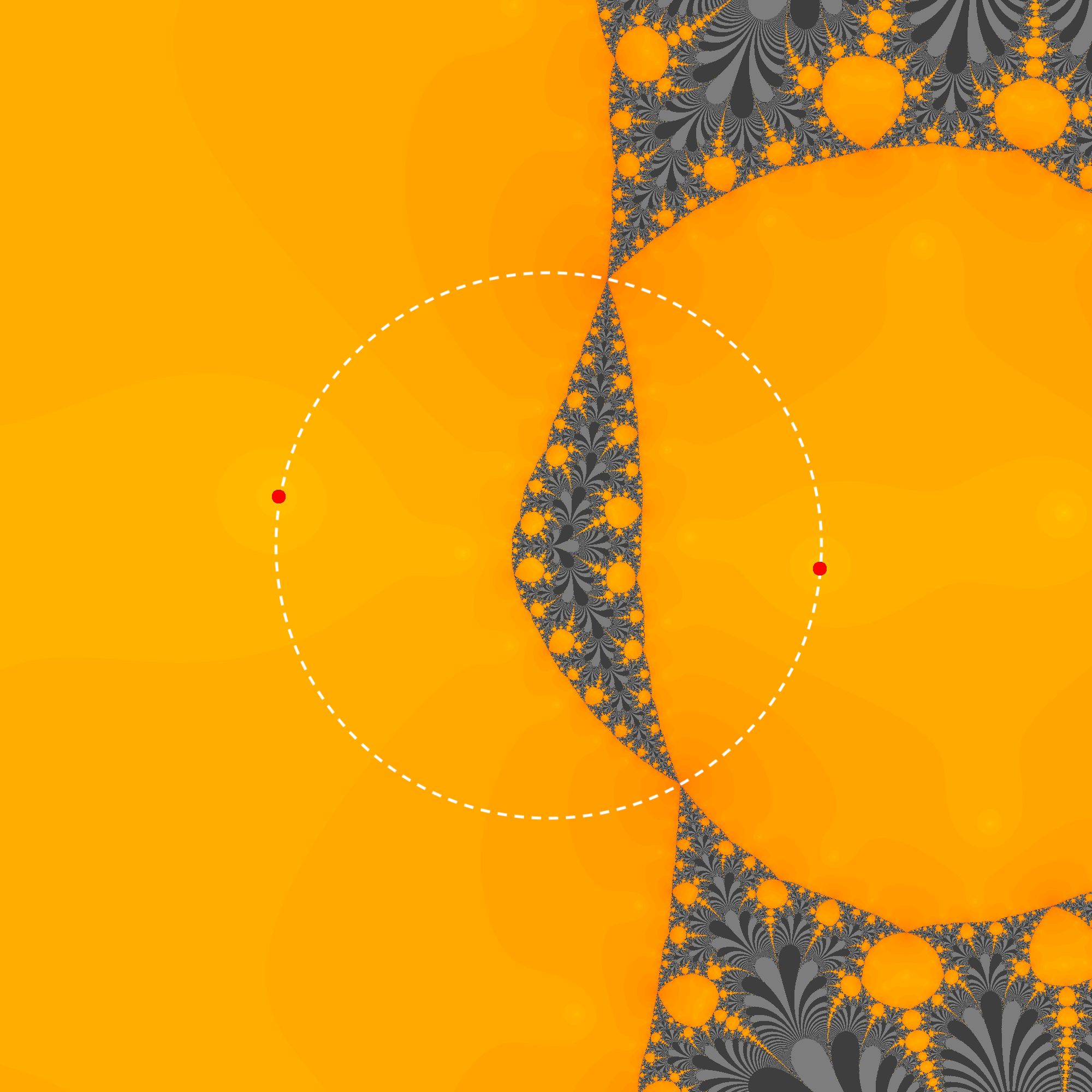}
  \caption{An illustration of the dynamics of the function $f_{\alpha,\beta}$ from Example~\ref{ex.arnold} with $\alpha=3.1$ and $\beta=0.8$. Escaping points are coloured in grey, and yellow points lie in the basins of attraction of an attracting cycle of period two (\reddot{$\bullet$}). The unit circle is invariant and has been drawn in white.}\label{f4}
\end{figure}

\begin{remark}\normalfont
Note that in the previous two examples, the set $I(f)\cup\{0,\infty\}$ is disconnected by a Jordan curve separating $0$ from $\infty$. One might ask if this is necessarily the case when $I(f) \cup \{0, \infty\}$ is disconnected; we do not provide an answer to this question. Observe that if $f$ is a transcendental self-map $f$ of $\C^*$, unless~$f$~has a Herman ring, there is at most one invariant curve that separates $0$ from $\infty$. 
\end{remark}
It is natural to ask if the set $I_e(f)$ can be connected for some $e \in \{0, \infty\}^{\N_0}$. Our goal in the next example is to answer this question in the affirmative by showing that there is a transcendental self-map $f$ of $\C^*$ such that $I_\infty(f)$ is connected. We will also show that for this function, $I(f)$ is connected, and hence $f$ is of type~\ref{I1}. We first prove the following general proposition that will be used in the example; this is based on the proof in \cite{Fast} that the escaping set of the map $z \mapsto \cosh^2 z$ is connected. Recall that if $f$ is a transcendental entire function or a transcendental self-map of $\C^*$, we say that a set $X$ is \textit{backward invariant} under $f$ if $f^{-1}(X) \subset X$.

\begin{proposition}
\label{prop:conn}
Let $f$ be a transcendental entire function such that $J(f) = \C$, or a transcendental self-map of $\C^*$ such that $J(f) = \C^*$. Furthermore, suppose that $f$ has no asymptotic values and only finitely many critical values. If $X$ is a backward invariant set under $f$ and $E \subset X$ is path-connected, not a singleton, and has the property that $E$ meets every component of $f^{-1}(E)$, then $X$ is connected.
\end{proposition}
\begin{proof}
We consider only the case of a transcendental entire function; the other case is almost identical. Observe that since $E$ is path-connected and $f$ has no asymptotic values and only finitely many critical values, every component of $f^{-1}(E)$ is path-connected. Since $E$ meets every component of $f^{-1}(E)$, then $E \cup f^{-1}(E)$ is path-connected. Since $X$ is backward invariant, this is in fact a subset of $X$. By repeated application of this argument, we deduce that $T = \bigcup_{n \geqslant 0} f^{-n}(E)$ is a connected subset of $X$.

Since $J(f)$ is the whole plane, and the backward orbit of any non-exceptional point is dense in $J(f)$, it follows that $\overline{T} = \C$. Hence $T \subset X \subset \overline{T}$, and so $X$ is connected.
\end{proof}

We now give the example.

\begin{example}\normalfont
\label{ex.5}
Set $\tilde{f}(z) \defeq \cosh z$, which is a transcendental entire function that is a lift of the transcendental self-map of $\C^*$ given by
$$
f(z) \defeq \exp\left(0.5(z + 1/z)\right)
$$
Observe that $J(\tilde{f})=\C$. Indeed, $\tilde{f}$ has only two singular values, which are the critical values at $\pm 1$. So $\tilde{f}\in\mathcal S$ and hence $\tilde{f}$ has no Baker domains and no wandering domains. Moreover, it follows from the classification of Fatou components \cite[Section~4.2]{bergweiler93} that since all the singular values of $\tilde{f}$ escape to $\infty$, $\tilde{f}$ cannot have any other type of Fatou component. Berweiler \cite{bergweiler95} showed that if $\tilde{f}$ is a lift of a transcendental self-map $f$ of $\C$, then $J(\tilde{f})=\exp^{-1}(J(f))$, and therefore $J(f)=\C^*$.

Define the set $$E \defeq i\mathbb{R}\cup \bigcup_{n \in \Z} \{ x + iy \in \C : y = n \pi \}.$$ It can be shown that, with $X = I(\tilde{f})$, the hypotheses of Proposition~\ref{prop:conn} are satisfied; indeed, it is a calculation that $E \subset I(\tilde{f})$ and that $E$ contains every preimage of the real line. Hence $I(\tilde{f})$ is connected.

Now put $X = I_\infty(f)$ and take $\exp E$ in place of $E$, which again satisfy the hypotheses of the version for $\C^*$ of Proposition~\ref{prop:conn}, and so the set $I_\infty(f)$ is connected. Note that it also follows from Proposition~\ref{prop:conn} that $I(f)$ is connected. Observe that $I_e(f)$ is not a $\C^*$-spider's web for any $e\in\{0,\infty\}^{\N_0}$, since there are infinitely many sequences $e'\in \{0,\infty\}^{\N_0}\setminus \{e\}$ for which $A_{e'}(f)\cap I_e(f) = \emptyset$ and, by Lemma~\ref{lem:fast-escaping-set} (iii), all the components of $A_{e'}(f)$ are unbounded in $\C^*$.
\exend
\end{example}

\section{Doubly connected Baker domains}
\label{sec:baker-domains}

We conclude the paper with our results concerning doubly connected Baker domains. First, we prove Theorem~\ref{thm:baker-domain}, by constructing the first example of such a domain. To that end, we use the following result from approximation theory (see \cite[Corollary in p.~162]{gaier87}).
\begin{lemma}
\label{lem:approx}
Suppose that $S\subseteq \C$ is a closed set such that $\widehat{\C}\setminus S$ is connected and locally connected at $\infty$, and assume that $S$ lies in a sector
$$
W_\alpha\defeq\{z\in\C\ :\ |\textup{arg}\,z|\leqslant \alpha/2\},
$$
for some $0<\alpha\leqslant 2\pi$. Suppose that $\varepsilon(r)$ is a real function that is continuous and positive for $r\geqslant 0$ and satisfies 
$$
\int_1^{+\infty} r^{-(\pi/\alpha)-1}\log \varepsilon(r) \ dr >-\infty.
$$
If $g_0 : S \to \C$ is continuous on $S$ and holomorphic on the interior of $S$, then there exists an entire function $g$ such that
\begin{equation}
\label{eq:gdef}
|g(z)-g_0(z)|<\varepsilon(|z|),\quad \text{ for all } z\in S. 
\end{equation}
\end{lemma}

This result was used in \cite[Lemma 4.5]{martipete3} to construct transcendental self-maps of $\C^*$ with Baker domains and wandering domains that escape in any possible way (see \cite[Theorem 1.1 and Theorem 1.2]{martipete3}). However, it is not clear if the Baker domains constructed in \cite{martipete3} are simply or doubly connected. 

\begin{proof}[Proof of Theorem~\ref{thm:baker-domain}]
Choose $R>0$ sufficiently large that both
\begin{equation}
\label{eq:Rdef}
\left|\exp\left(\epsilon+z^{-2}\right) - 1\right| \leqslant 4|z|^{-2}, \qfor |\epsilon| \leqslant |z|^{-2} \leqslant \frac{4}{R^2},
\end{equation}
and also
\begin{equation}
\label{eq:Rdef2}
R \geqslant 8|z|^{-1} + 12R|z|^{-2}, \qfor |z| \geqslant \frac{R}{2}.
\end{equation}
Define the set
\[
S  \defeq \{ z \in \C : |z - 3R| \leqslant R \} \cup \{ z \in \C : \operatorname{Re} z \geqslant 3R \},
\]
and let $g_0$ denote the principal branch of the logarithm on $S\subseteq \C\setminus (-\infty,0]$. 

Next set 
\[
\varepsilon(r) \defeq 
\begin{cases} r^{-2}, &\text{for } r > 1, \\
              1     , &\text{for } 0 \leqslant r \leqslant 1.
\end{cases}
\]
Note that $S \subset W_{\pi}$, and
\[
\int_1^{\infty} r^{-2} \log \varepsilon(r) \ dr = -2 \int_1^{\infty} r^{-2} \log r \ dr = -2 > -\infty.
\]
It follows, by Lemma~\ref{lem:approx}, that there exists a transcendental entire function $g$ such that \eqref{eq:gdef} holds.

Now set 
\[
f(z) \defeq \exp(g(z + 3R) + z^{-2}),
\]
so that $f$ is a transcendental self-map of $\C^*$. Let 
$$
S' \defeq \{ z \in \C : R/2 < |z| \text{ and } z + 3R \in S \}.
$$
(see Figure~\ref{fig:g0}). Suppose that $z \in S'$. Then, by \eqref{eq:gdef}, \eqref{eq:Rdef} and \eqref{eq:Rdef2}, we have that
\begin{align*}
\operatorname{Re}(f(z))
     &= \operatorname{Re}(\exp[g(z + 3R) + z^{-2}]), \\
     &= \operatorname{Re}(\exp[g_0(z + 3R)] \cdot \exp[g(z + 3R) - g_0(z + 3R) + z^{-2}]), \\
		 &= \operatorname{Re}((z + 3R) \cdot \exp[g(z + 3R) - g_0(z + 3R) + z^{-2}]), \\
		 &\geqslant (\operatorname{Re}(z) + 3R) \cdot (1 - 4|z|^{-2}) - \left|\operatorname{Im}(z)\right| \cdot 4|z|^{-2}, \\
		 &\geqslant \operatorname{Re}(z) + 3R - 8|z|^{-1} - 12R|z|^{-2}, \\
		 &\geqslant \operatorname{Re}(z) + 2R.
\end{align*}
		 
\begin{figure}[ht!]
\centering
\def\svgwidth{.4\linewidth}
\input{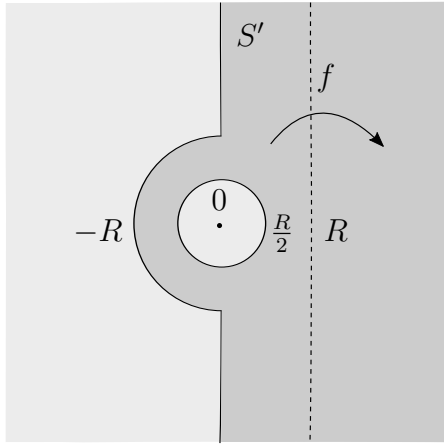}
\caption{The set $S'$ in the proof of Theorem~\ref{thm:baker-domain}.}
\label{fig:g0}
\end{figure}

Thus, $z \in S'$ implies that $\operatorname{Re} f(z) \geqslant \operatorname{Re} z + 2R$ and hence also that $f(z) \in S'$. We can deduce that $S'$ lies in an invariant Baker domain of $f$. Since $S'$ contains the circle $\{ z \in \C : |z| = 3R/4 \}$, the Baker domain must be doubly connected.
\end{proof}

Suppose that $U$ is a Baker domain of a transcendental meromorphic function or trans\-cendental self-map $f$ of $\C^*$. We say that a set $H\subseteq U$ is an \textit{absorbing set} if $f(H)\subseteq H$ and for every compact set $K\subseteq U$, there exists $n\in\N$ such that $f^n(K)\in H$. The existence of simply connected absorbing sets for Baker domains of transcendental entire functions was established by Cowen \cite{cowen}. In \cite{bfjk15}, the authors study when a Baker domain of a transcendental meromorphic functions admits a simply connected absorbing domain. Even though Baker domains can be doubly connected in $\C^*$ as we have seen in Theorem~\ref{thm:baker-domain}, it is easy to show that they always contain a simply connected absorbing domain.

\begin{lemma}
\label{lem:abs-dom}
Let $f$ be a transcendental self-map of $\C^*$ such that $f$ has a Baker domain $U$. Then $U$ contains a simply connected absorbing set.
\end{lemma}
\begin{proof}
By  \cite[Lemma~3.5]{martipete3}, we can find a suitable lift $\tilde{f}$ of $f$ for which a component $V$ of $\exp^{-1}(U)$ is a Baker domain for $\tilde{f}$. Then, it follows from the classification of Baker domains for transcendental entire functions that the Baker domain $V$ of $\tilde{f}$ contains an absorbing set $H'$ \cite[Theorem~5.1]{rippon08}. Then, the set $H=\exp(H')$ is an absorbing set for $U$. Indeed, for every $z\in U$, there exists a point $w\in \exp^{-1}(z)$ that lies in $V$ and if $n\in\N$ is such that $\tilde{f}^n(w)\in H'$, then $f^n(z)\in H$, as we wanted to show.  
\end{proof}

Observe that this means that we can transfer the classification of Baker domains for transcendental entire functions (see, for example, \cite[Section~5]{rippon08}) to transcendental self-maps of $\C^*$ by using the fact that $f$ and any of its lifts $\tilde{f}$ are conjugated on the absorbing set.

In the case that $f$ has a doubly connected Baker domain, we can deduce some additional properties of $f$. Recall that given a transcendental self-map $f$ of $\C^*$, we define the \textit{index of} $f$ as the index (or winding number) of $f(\gamma)$ with respect to $0$, where $\gamma\subseteq \C^*$ is any positively oriented simple closed curve around $0$. This quantity is a topological invariant of $f$. We prove that $\text{ind}(f)=0$ when $f$ has a doubly connected Baker domain.

\begin{lemma}
\label{lem:index}
Let $f$ be a transcendental self-map of $\C^*$ such that $f$ has a doubly connected Baker domain. Then $\textup{ind}(f)=0$.
\end{lemma}
\begin{proof}
Let $U$ be the doubly connected Baker domain of $f$ and suppose that $H\subseteq U$ is a simply connected absorbing set. Suppose that $\gamma\subseteq U$ is a positively oriented simple closed curve around $0$. Let $n\in\N$ be such that $f^n(\gamma)\subseteq H$. Then, 
$$
0=\text{ind}(f^n(\gamma),0)=\text{ind}(\gamma,0)\cdot \text{ind}(f)^n=\text{ind}(f)^n,
$$
and so $\text{ind}(f)=0$ as required.
\end{proof}

We conclude the paper with the proof of Theorem~\ref{thm:bd-properties}, that relates the fact that $f$ has a doubly connected Baker domain with the connectivity of the set $I(f)\cup \{0,\infty\}$. 

\begin{proof}[Proof of Theorem~\ref{thm:bd-properties}]
Let $U$ be the doubly connected Baker domain of $f$ and, by Lemma~\ref{lem:abs-dom}, let $H\subseteq U$ be a simply connected absorbing set. By taking a suitable iterate of $f$, we may assume without loss of generality that $U$ is invariant. Consider $H_n$ to be the component of $f^{-n}(H)$ that contains $H$. Then, we can write
$$
U=\bigcup_{n\in \N} H_n,
$$
where $H_n\subseteq H_{n+1}$ for $n\in \N$. Since $U$ is doubly connected, there exists $n_0\in\N$ such that $H_n$ is doubly connected and $n_0$ is minimal with this property; note that the union of an increasing sequence of open simply connected sets is simply connected. We claim that the closure $\widehat{H}_{n_0+1}$ necessarily contains both $0$ and $\infty$. Indeed, let $\gamma\subseteq H_{n_0}$ be a curve of index $1$ around $0$. Each of the components of the preimage of $\gamma$ in $H_{n_0+1}$ is a cuve $\gamma'$ that is unbounded in $\C^*$. Observe that each of the complementary components of $H_{n_0}$ in $U$ must contain a component of the preimage of $\gamma$\red{, and so the closure of $U$ in $\widehat{\mathbb{C}}$ contains $\{0, \infty\}$.} This proves the claim and the theorem.
\end{proof}

\bibliographystyle{amsalpha}

\bibliography{EscapingSet}

\end{document}